\documentclass[12pt]{amsart}
\usepackage{amsmath,amsthm,amsfonts,amssymb,mathrsfs}
\date{\today}
\input xymatrix
\xyoption{all}
\usepackage{color}

\newtheorem{theorem}{Theorem}[section]
\newtheorem{proposition}[theorem]{Proposition}
\newtheorem{corollary}[theorem]{Corollary}
\newtheorem{lemma}[theorem]{Lemma}
\theoremstyle{definition}
\newtheorem{example}[theorem]{Example}
\newtheorem{remark}[theorem]{Remark}

\usepackage{hyperref}

\begin{document}

\title[On monoids of monotone injective partial
selfmaps of integers ]
{On monoids of monotone injective partial self\-maps of integers
with cofinite domains and images}

\author[O.~Gutik]{Oleg~Gutik}
\address{Department of Mechanics and Mathematics, Ivan Franko Lviv
National University, Universytetska 1, Lviv, 79000, Ukraine}
\email{o\_gutik@franko.lviv.ua, ovgutik@yahoo.com}

\author[D.~Repov\v{s}]{Du\v{s}an~Repov\v{s}}
\address{Faculty of Education, and
Faculty of Mathematics and Physics, University of Ljubljana,
P.~O.~B. 2964, Ljubljana, 1001, Slovenia}
\email{dusan.repovs@guest.arnes.si}



\begin{abstract}
We study the semigroup
$\mathscr{I}^{\!\nearrow}_{\infty}(\mathbb{Z})$ of monotone
injective partial selfmaps of the set of integers having cofinite
domain and image. We show that
$\mathscr{I}^{\!\nearrow}_{\infty}(\mathbb{Z})$ is bisimple and all
of its non-trivial semigroup homomorphisms are either isomorphisms
or group homomorphisms. We also prove that every Baire topology
$\tau$ on $\mathscr{I}^{\!\nearrow}_{\infty}(\mathbb{Z})$ such that
$(\mathscr{I}^{\!\nearrow}_{\infty}(\mathbb{Z}),\tau)$ is a
Hausdorff semitopological semigroup is discrete and we construct a
non-discrete Hausdorff inverse semigroup topology $\tau_W$ on
$\mathscr{I}^{\!\nearrow}_{\infty}(\mathbb{Z})$. We show that the
discrete semigroup $\mathscr{I}^{\!\nearrow}_{\infty}(\mathbb{Z})$
cannot be embedded into some classes of compact-like topological
semigroups and that its remainder under the closure in a topological
semigroup $S$ is an ideal in $S$.
\end{abstract}

\maketitle


\section{Introduction and preliminaries}

In this paper all spaces will be assumed to be Hausdorff. We shall
denote the first infinite cardinal by $\omega$ and the cardinality
of the set $A$ by $|A|$. Also,
we shall denote the additive group of
integers by $\mathbb{Z}(+)$. We shall identify all sets $X$ with 
their
cardinality $|X|$.

For a topological space $X$, a family $\{A_s\mid s\in\mathscr{A}\}$
of subsets of $X$ is called \emph{locally finite} if for every point
$x\in X$ there exists an open neighbourhood $U$ of $x$ in $X$ such
that the set $\{s\in\mathscr{A}\mid U\cap A_s\neq\varnothing\}$ is finite. A subset 
$A$ of $X$ is said to be
\begin{itemize}
  \item \emph{co-dense} in $X$ if $X\setminus A$ is dense in $X$;
  \item an \emph{$F_\sigma$-set} in $X$ if $A$ is an intersection of
   a countable family of open subsets in $X$.
\end{itemize}

We recall that a topological space $X$ is said to be
\begin{itemize}
  \item \emph{compact} if each open cover of $X$ has a finite
   subcover;
  \item \emph{countably compact} if each open countable cover of
   $X$ has a finite subcover;
  \item \emph{pseudocompact} if each locally finite open cover of
  $X$ is finite;
  \item a \emph{Baire space} if for each
sequence $A_1, A_2,\ldots, A_i,\ldots$ of nowhere dense subsets of
$X$ the union $\displaystyle\bigcup_{i=1}^\infty A_i$ is a co-dense
subset of $X$;
  \item  \emph{\v{C}ech
   complete} if $X$ is Tychonoff and for every compactification
   $cX$ of $X$ the remainder $cX\setminus X$ is an $F_\sigma$-set
   in $cX$;
  \item  \emph{locally compact} if every point of $X$ has an open
   neighbourhood with a
   compact closure.
\end{itemize}
According to Theorem~3.10.22 of \cite{Engelking1989}, a Tychonoff
topological space $X$ is pseudocompact if and only if each
continuous real-valued function on $X$ is bounded.

An algebraic semigroup $S$ is called {\it inverse} if for any
element $x\in S$ there exists a unique $x^{-1}\in S$ such that
$xx^{-1}x=x$ and $x^{-1}xx^{-1}=x^{-1}$. The element $x^{-1}$ is
called the {\it inverse of} $x\in S$. If $S$ is an inverse
semigroup, then the function $\operatorname{inv}\colon S\to S$ which
assigns to every element $x$ of $S$ its inverse element $x^{-1}$ is
called an {\it inversion}.

A \emph{band} is a semigroup of idempotents. If $S$ is a semigroup, then we shall denote the subset of
idempotents in $S$ by $E(S)$. If $S$ is an inverse semigroup, then
$E(S)$ is closed under multiplication. The semigroup operation on $S$
determines the following partial order $\leqslant$ on $E(S)$:
$e\leqslant f$ if and only if $ef=fe=e$. This order is called the
{\em natural partial order} on $E(S)$. A \emph{semilattice} is a
commutative semigroup of idempotents. A semilattice $E$ is called
{\em linearly ordered} or a \emph{chain} if its natural order is a
linear order. A \emph{maximal chain} of a semilattice $E$ is a chain
which is properly contained in no other chain of $E$.

If $h\colon S\rightarrow T$ is a homomorphism from a semigroup $S$
into a semigroup $T$ then we say that $h$ is:
\begin{itemize}
  \item a \emph{trivial homomorphism} if $h$ is either an
   isomorphism or an annihilating homomorphism;
  \item a \emph{group homomorphism} if $(S)h$ is a subgroup of $T$.
\end{itemize}

If $\mathfrak{C}$ is an arbitrary congruence on a semigroup $S$,
then we denote by $\Phi_\mathfrak{C}\colon S\rightarrow
S/\mathfrak{C}$ the natural homomorphism from $S$ onto the quotient
semigroup $S/\mathfrak{C}$. A congruence $\mathfrak{C}$ on a
semigroup $S$ is called \emph{non-trivial} if $\mathfrak{C}$ is
distinct from universal and identity congruences on $S$, and a
\emph{group congruence} if the quotient semigroup $S/\mathfrak{C}$ is a group.
Every inverse semigroup $S$ admits a least (minimal) group
congruence $\mathfrak{C}_{mg}$:
\begin{equation*}
    a\mathfrak{C}_{mg}b \; \hbox{ if and only if there exists }\;
    e\in E(S) \; \hbox{ such that }\; ae=be
\end{equation*}
(see Lemma~III.5.2 of \cite{Petrich1984}).

The Axiom of Choice implies the existence of maximal chains in any
partially ordered set. According to \cite{Petrich1984},
a chain $L$
is called an $\omega$-chain if $L$ is isomorphic to
$\{0,-1,-2,-3,\ldots\}$ with the usual order $\leqslant$. Let $E$ be
a semilattice and $e\in E$. We denote ${\downarrow} e=\{ f\in E\mid
f\leqslant e\}$ and ${\uparrow} e=\{ f\in E\mid e\leqslant f\}$.  By
$(\mathscr{P}_{<\omega}(\lambda),\subseteq)$ we shall denote the
\emph{free semilattice with identity} over a set of cardinality
$\lambda\geqslant\omega$, i.e.,
$(\mathscr{P}_{<\omega}(\lambda),\subseteq)$ is the set of all
finite subsets (including
the empty set) of $\lambda$ with the
semilattice operation ``union''.

If $S$ is a semigroup, then we shall denote the {\it Green relations}
on
$S$ by $\mathscr{R}$, $\mathscr{L}$, $\mathscr{J}$, $\mathscr{D}$
and $\mathscr{H}$ (see Section~2.1 of \cite{CP}):
\begin{align*}
    &\qquad a\mathscr{R}b \mbox{ if and only if } aS^1=bS^1;\\
    &\qquad a\mathscr{L}b \mbox{ if and only if } S^1a=S^1b;\\
    &\qquad a\mathscr{J}b \mbox{ if and only if } S^1aS^1=S^1bS^1;\\
    &\qquad \mathscr{D}=\mathscr{L}\circ\mathscr{R}=
          \mathscr{R}\circ\mathscr{L};\\
    &\qquad \mathscr{H}=\mathscr{L}\cap\mathscr{R}.
\end{align*}

A semigroup $S$ is called \emph{simple} if $S$ does not contain
any
proper two-sided ideals and \emph{bisimple} if $S$ has only one
$\mathscr{D}$-class.

A {\it semitopological} (resp. \emph{topological}) {\it semigroup}
is a Hausdorff topological space together with a separately (resp.
jointly) continuous semigroup operation. An inverse topological
semigroup with the continuous inversion is called a
\emph{topological inverse semigroup}. A Hausdorff topology $\tau$ on
a (inverse) semigroup $S$ such that $(S,\tau)$ is a topological
(inverse) semigroup is called a (\emph{inverse}) \emph{semigroup
topology}.

If $\alpha\colon X\rightharpoonup Y$ is a partial map, then by
$\operatorname{dom}\alpha$ and $\operatorname{ran}\alpha$ we shall
denote the domain and the range of $\alpha$, respectively. Let
$\mathscr{I}_\lambda$ denote the set of all partial injective
transformations of an infinite set $X$ of cardinality $\lambda$
together with the following semigroup operation:
$x(\alpha\beta)=(x\alpha)\beta$ if
$x\in\operatorname{dom}(\alpha\beta)=\{
y\in\operatorname{dom}\alpha\mid
y\alpha\in\operatorname{dom}\beta\}$,  for
$\alpha,\beta\in\mathscr{I}_\lambda$. The semigroup
$\mathscr{I}_\lambda$ is called the \emph{symmetric inverse
semigroup} over the set $X$~(see Section~1.9 of \cite{CP}). The
symmetric inverse semigroup was introduced by
Wagner~\cite{Wagner1952} and it plays a major role in the theory of
semigroups.


Let $\mathbb{Z}$ be the set of integers with the usual order
$\leqslant$. We shall say that a partial map
$\alpha\colon\mathbb{Z}\rightharpoonup\mathbb{Z}$ is \emph{monotone}
if $n\leqslant m$ implies $(n)\alpha\leqslant(m)\alpha$ for
$n,m\in\mathbb{Z}$. By $\mathscr{I}^\nearrow_{\infty}(\mathbb{Z})$
we denote a subsemigroup of injective partial monotone selfmaps of
$\mathbb{Z}$ with cofinite domains and images, i.e.,
\begin{equation*}
 \mathscr{I}^{\!\nearrow}_{\infty}(\mathbb{Z})=
 \left\{\alpha\in\mathscr{I}_\omega\mid
 \alpha \mbox{ is monotone},
 |\mathbb{Z}\setminus\operatorname{dom}\alpha|<\infty
 \quad \mbox{and} \quad
 |\mathbb{Z}\setminus\operatorname{ran}\alpha|<\infty\right\}.
\end{equation*}
Obviously, $\mathscr{I}^{\!\nearrow}_{\infty}(\mathbb{Z})$ is an
inverse submonoid of the semigroup $\mathscr{I}_\omega$. We observe
that $\mathscr{I}^{\!\nearrow}_{\infty}(\mathbb{Z})$ is a countable
semigroup. Furthermore, we shall denote the identity of the
semigroup $\mathscr{I}^{\!\nearrow}_{\infty}(\mathbb{Z})$ by
$\mathbb{I}$ and the group of units of
$\mathscr{I}^{\!\nearrow}_{\infty}(\mathbb{Z})$ by $H(\mathbb{I})$.

\begin{lemma}\label{lemma-1.1}
A partial injective monotone map $\alpha$ is an
element of the semigroup
$\mathscr{I}^{\!\nearrow}_{\infty}(\mathbb{Z})$ if and only if there
exist integers $d_\alpha$ and $u_\alpha$ such that the following
conditions hold:
\begin{equation*}
    (m-1)\alpha=(m)\alpha-1 \quad \mbox{and} \quad
    (n+1)\alpha=(n)\alpha+1 \quad \mbox{for all integers }
    m\leqslant d_\alpha \mbox{ and } n\geqslant u_\alpha,
\end{equation*}
and $\alpha\in H(\mathbb{I})$ if and only if
$(n+1)\alpha=(n)\alpha+1$ for any integer $n$.
\end{lemma}

\begin{proof}
The implication $(\Leftarrow)$ 
is trivial.

$(\Rightarrow)$ Since for every element $\alpha$ of the semigroup
$\mathscr{I}^{\!\nearrow}_{\infty}(\mathbb{Z})$ the sets
$\mathbb{Z}\setminus\operatorname{dom}\alpha$ and
$\mathbb{Z}\setminus\operatorname{ran}\alpha$ are finite,
we conclude
that there exist integers $d_\alpha$ and $u_\alpha$ such that the
following condition holds:
\begin{equation}\label{eq-1.1}
  m,n\in\operatorname{dom}\alpha\cap\operatorname{ran}\alpha, \quad
  \hbox{ for all } m\leqslant d_\alpha \hbox{ and } n\geqslant
  u_\alpha.
\end{equation}
Now, since the partial map $\alpha\colon \mathbb{Z}\rightharpoonup
\mathbb{Z}$ is monotone we have that
\begin{equation*}
    (m-1)\alpha\leqslant(m)\alpha-1 \quad \hbox{ and }\quad
    (n)\alpha+1\leqslant(n+1)\alpha \quad \hbox{for all }m\leqslant
    d_\alpha \hbox{ and } n\geqslant u_\alpha,
\end{equation*}
and hence we get that
\begin{equation*}
\begin{split}
  (m-j)\alpha & \leqslant(m-(j-1))\alpha-1\leqslant\cdots
    \leqslant(m)\alpha-j \quad \hbox{ and }\quad \\
  (n)\alpha+j&\leqslant\cdots\leqslant(n+j-1)\alpha+1 \leqslant(n+j)\alpha
\end{split}
\end{equation*}
for any positive integer $j$, $m\leqslant d_\alpha$ and $n\geqslant
u_\alpha$. Then by condition (\ref{eq-1.1}) we have that
\begin{equation*}
    (m-1)\alpha=(m)\alpha-1 \quad \mbox{and} \quad
    (n+1)\alpha=(n)\alpha+1 \quad \mbox{for all integers }
    m\leqslant d_\alpha \mbox{ and } n\geqslant u_\alpha.
\end{equation*}

It is obvious that if $(n+1)\alpha=(n)\alpha+1$ for any integer $n$
then $\alpha\colon \mathbb{Z}\rightarrow\mathbb{Z}$ is a bijective
monotone map and hence $\alpha\in H(\mathbb{I})$. Conversely, if
$\alpha\in H(\mathbb{I})$ then $\alpha\colon
\mathbb{Z}\rightharpoonup\mathbb{Z}$ is a bijective monotone map and
the first assertion of lemma implies that $(n+1)\alpha=(n)\alpha+1$
for any integer $n$.
\end{proof}

The bicyclic semigroup ${\mathscr{C}}(p,q)$ is the semigroup with
the identity $1$ generated by elements $p$ and $q$ subject only to
the condition $pq=1$. The bicyclic semigroup is bisimple and every
one of its congruences is either trivial or a group congruence.
Moreover, every non-annihilating homomorphism $h$ of the bicyclic
semigroup is either an isomorphism or the image of
${\mathscr{C}}(p,q)$ under $h$ is a cyclic group~(see Corollary~1.32
in \cite{CP}).

The bicyclic semigroup plays an important role in the algebraic
theory of semigroups and in the theory of topological semigroups.
For example, the well-known result of Andersen~\cite{Andersen}
states that a ($0$--)simple semigroup is completely ($0$--)simple if
and only if it does not contain the bicyclic semigroup.

\begin{remark}\label{remark-1.2}
Let $n$ be an arbitrary integer. We put $\mathscr{C}(n,+)$ and
$\mathscr{C}(n,-)$ to be semigroups which are generated by partial
transformations $\alpha_n^+$ and $\beta_n^+$; $\alpha_n^-$ and
$\beta_n^-$, respectively, of the set of integers $\mathbb{Z}$,
defined as follows:
\begin{align*}
    (i)\alpha_n^+=
    \left\{
      \begin{array}{cl}
        i, & \hbox{if } i\leqslant n; \\
        i+1, & \hbox{if } i> n,
      \end{array}
    \right.
\qquad
    (i)\beta_n^+=
    \left\{
      \begin{array}{cl}
        i, & \hbox{if } i\leqslant n; \\
        i-1, & \hbox{if } i> n+1,
      \end{array}
    \right.
\\
    (i)\alpha_n^-=
    \left\{
      \begin{array}{cl}
        i, & \hbox{if } i\geqslant n; \\
        i-1, & \hbox{if } i< n,
      \end{array}
    \right.
 \qquad
    (i)\beta_n^-=
    \left\{
      \begin{array}{cl}
        i, & \hbox{if } i\geqslant n; \\
        i+1, & \hbox{if } i< n-1,
      \end{array}
    \right.
\end{align*}
$i\in\mathbb{Z}$. We remark that $\mathscr{C}(n,+)$ and
$\mathscr{C}(n,-)$ are bicyclic semigroups for every positive
integer $n$. Therefore the semigroup
$\mathscr{I}_{\infty}^{\!\nearrow}(\mathbb{Z})$ contains infinitely
many isomorphic copies of the bicyclic semigroup
${\mathscr{C}}(p,q)$.
\end{remark}

We shall say that a partial map
$\alpha\colon\mathbb{Z}\rightharpoonup\mathbb{Z}$ is \emph{almost
monotone} if there exists a finite subset $F$ in
$\operatorname{dom}\alpha$ such that the restriction
$\alpha|_{\operatorname{dom}\alpha\setminus F}\colon
\mathbb{Z}\rightharpoonup\mathbb{Z}$ is a monotone partial map. By
$\mathscr{I}^\looparrowright_{\infty}(\mathbb{Z})$ we denote a
subsemigroup of injective partial almost monotone selfmaps of
$\mathbb{Z}$ with cofinite domains and images, i.e.,
\begin{equation*}
 \mathscr{I}^{\looparrowright}_{\infty}(\mathbb{Z})=
 \left\{\alpha\in\mathscr{I}_\omega\mid
 \alpha \mbox{ is almost monotone},
 |\mathbb{Z}\setminus\operatorname{dom}\alpha|<\infty
 \quad \mbox{and} \quad
 |\mathbb{Z}\setminus\operatorname{ran}\alpha|<\infty\right\}.
\end{equation*}
Obviously, $\mathscr{I}^{\looparrowright}_{\infty}(\mathbb{Z})$ is
an inverse submonoid of the semigroup $\mathscr{I}_\omega$ and
$\mathscr{I}^{\!\nearrow}_{\infty}(\mathbb{Z})$ is an inverse
submonoid of $\mathscr{I}^{\looparrowright}_{\infty}(\mathbb{Z})$.
We observe that $\mathscr{I}^{\looparrowright}_{\infty}(\mathbb{Z})$
is a countable semigroup.

It is well known that topological algebra studies the influence of
topological properties of its objects on their algebraic properties
and the influence of algebraic properties of its objects on their
topological properties. There are two main problems in topological
algebra: the problem of non-discrete topologization and the problem
of embedding into objects with some topological-algebraic
properties.

In mathematical literature the question about non-discrete
(Hausdorff) topologization was posed by Markov \cite{Markov1945}.
Pontryagin gave well known conditions for
a base at the unity of a group
for its non-discrete topologization (see Theorem~4.5 of
\cite{HewittRoos1963} or Theorem~3.9 of \cite{Pontryagin1966}).
Various authors have refined Markov's question: can a given infinite
group $G$ endowed with a non-discrete group topology be embedded
into a compact topological group? Again, for an arbitrary Abelian
group $G$ the answer is affirmative, but there is a non-Abelian
topological group that cannot be embedded into any compact
topological group ({see Section~9 of \cite{HBSTT}}).

Also, Ol'shanskiy \cite{Olshansky1980} constructed an infinite
countable group $G$ such that every Hausdorff group topology on $G$
is discrete. Eberhart and Selden  \cite{EberhartSelden1969}
showed 
that every Hausdorff semigroup topology on the bicyclic semigroup
$\mathscr{C}(p,q)$ is discrete. Bertman and West 
\cite{BertmanWest1976} proved that every Hausdorff topology $\tau$ on
$\mathscr{C}(p,q)$ such that $(\mathscr{C}(p,q),\tau)$ is a
semitopological semigroup is also discrete. Taimanov 
\cite{Taimanov1975} gave sufficient conditions on a commutative semigroup
to have a non-discrete semigroup topology.

Many mathematiciants have studied the problems of embeddings of
topological semigroups into compact or compact-like topological
semigroups (see \cite{CHK}). Neither stable nor $\Gamma$-compact
topological semigroups can contain a copy of the bicyclic
semigroup~\cite{AHK, HildebrantKoch1986}. Also, the bicyclic
semigroup cannot be embedded into any countably compact topological
inverse semigroup~\cite{GutikRepovs2007}. Moreover, the conditions
were given in \cite{BanakhDimitrovaGutik2009} and
\cite{BanakhDimitrovaGutik2010} when a countably compact or
pseudocompact topological semigroup cannot contain the bicyclic
semigroup.

However, Banakh, Dimitrova and Gutik~\cite{BanakhDimitrovaGutik2010}
have constructed (assuming the Continuum Hypothesis or the
Martin
Axiom) an example of a Tychonoff countably compact topological
semigroup which contains the bicyclic semigroup. The problems of
topologization of semigroups of partial transformations and their
embeddings into compact-like semigroup were studied in
\cite{GutikPavlyk2005, GutikPavlykReiter2009}.

We showed in \cite{GutikRepovs2011} that the semigroup
$\mathscr{I}_{\infty}^{\!\nearrow}(\mathbb{N})$ of partial cofinite
monotone injective transformations of the set of positive integers
$\mathbb{N}$ has algebraic properties similar to those of the
bicyclic semigroup: it is bisimple and all of its non-trivial
semigroup homomorphisms are either isomorphisms or group
homomorphisms. We proved that every locally compact topology $\tau$
on $\mathscr{I}_{\infty}^{\!\nearrow}(\mathbb{N})$ such that
$(\mathscr{I}_{\infty}^{\!\nearrow}(\mathbb{N}),\tau)$ is a
topological inverse semigroup, is discrete and we described the
closure of $(\mathscr{I}_{\infty}^{\!\nearrow}(\mathbb{N}),\tau)$ in
a topological semigroup.

In this paper we shall describe Green relations on
$\mathscr{I}^{\!\nearrow}_{\infty}(\mathbb{Z})$, show that
$\mathscr{I}^{\!\nearrow}_{\infty}(\mathbb{Z})$ is bisimple and all
of its non-trivial semigroup homomorphisms are either isomorphisms
or group homomorphisms. We shall also prove that every Baire
topology $\tau$ on $\mathscr{I}^{\!\nearrow}_{\infty}(\mathbb{Z})$
such that $(\mathscr{I}^{\!\nearrow}_{\infty}(\mathbb{Z}),\tau)$ is
a Hausdorff semitopological semigroup is discrete and construct a
non-discrete Hausdorff semigroup inverse topology $\tau_W$ on
$\mathscr{I}^{\!\nearrow}_{\infty}(\mathbb{Z})$. We shall show that
the discrete semigroup
$\mathscr{I}^{\!\nearrow}_{\infty}(\mathbb{Z})$ cannot be embedded
into some classes of compact-like topological semigroups and that
its remainder under the closure in a topological semigroup $S$ is an
ideal in $S$.


\section{Algebraic properties of the semigroup
$\mathscr{I}^{\!\nearrow}_{\infty}(\mathbb{Z})$}\label{section-2}

\begin{proposition}\label{proposition-2.1}
\begin{itemize}
    \item[$(i)$] An element $\alpha$ of the semigroup
         $\mathscr{I}^{\!\nearrow}_{\infty}(\mathbb{Z})$
         is an idempotent if and only if $(x)\alpha=x$ for every
         $x\in\operatorname{dom}\alpha$.

    \item[$(ii)$] If $\varepsilon,\iota\in
          E(\mathscr{I}^{\!\nearrow}_{\infty}(\mathbb{Z}))$,
          then $\varepsilon\leqslant\iota$ if and only if
          $\operatorname{dom}\varepsilon\subseteq
          \operatorname{dom}\iota$.

    \item[$(iii)$] The semilattice
          $E(\mathscr{I}^{\!\nearrow}_{\infty}(\mathbb{Z}))$ is
          isomorphic to
          $(\mathscr{P}_{<\omega}(\mathbb{Z}),\subseteq)$ under
          the mapping $(\varepsilon)h=\mathbb{Z}\setminus
          \operatorname{dom}\varepsilon$.

    \item[$(iv)$] Every maximal chain in
          $E(\mathscr{I}^{\!\nearrow}_{\infty}(\mathbb{Z}))$ is an
          $\omega$-chain.

    \item[$(v)$] $\alpha\mathscr{R}\beta$ in
         $\mathscr{I}^{\!\nearrow}_{\infty}(\mathbb{Z})$ if and only if
         $\operatorname{dom}\alpha=\operatorname{dom}\beta$.

    \item[$(vi)$] $\alpha\mathscr{L}\beta$ in
         $\mathscr{I}^{\!\nearrow}_{\infty}(\mathbb{Z})$ if and only if
         $\operatorname{ran}\alpha=\operatorname{ran}\beta$.

    \item[$(vii)$] $\alpha\mathscr{H}\beta$ in
         $\mathscr{I}^{\!\nearrow}_{\infty}(\mathbb{Z})$ if and only if
         $\operatorname{dom}\alpha=\operatorname{dom}\beta$ and
         $\operatorname{ran}\alpha=\operatorname{ran}\beta$.

    \item[$(viii)$] $\mathscr{I}^{\!\nearrow}_{\infty}(\mathbb{Z})$ is a
         simple semigroup and hence
         $\mathscr{J}=\mathscr{I}^{\!\nearrow}_{\infty}(\mathbb{Z})\times
         \mathscr{I}^{\!\nearrow}_{\infty}(\mathbb{Z})$.

     \item[$(ix)$] For all idempotents $\varepsilon,\varphi\in
         \mathscr{I}^{\!\nearrow}_{\infty}(\mathbb{Z})$ there exist
         infinitely many elements $\alpha,\beta\in
         \mathscr{I}^{\!\nearrow}_{\infty}(\mathbb{Z})$ such that
         $\alpha\cdot\beta=\varepsilon$ and
         $\beta\cdot\alpha=\varphi$.
\end{itemize}
\end{proposition}

\begin{proof}
Statements $(i)-(iv)$ are trivial and they follow from the
definition of the semigroup
$\mathscr{I}^{\!\nearrow}_{\infty}(\mathbb{Z})$.

The proofs of $(v)-(vii)$ follow trivially from the fact that
$\mathscr{I}^{\!\nearrow}_{\infty}(\mathbb{Z})$ is a regular
semigroup, and Proposition 2.4.2 and Exercise 5.11.2 in
\cite{Howie1995}.

$(viii)$ Note that every cofinite subset of $\mathbb{Z}$ is
order-isomorphic to $\mathbb{Z}$. Let
$\varphi,\gamma\in\mathscr{I}^{\!\nearrow}_{\infty}(\mathbb{Z})$ be
arbitrary. Since the sets $\mathbb{Z}\setminus
\operatorname{dom}\varphi$, $\mathbb{Z}\setminus
\operatorname{dom}\gamma$ and $\mathbb{Z}\setminus
\operatorname{ran}\varphi$ are finite and the sets
$\operatorname{dom}\varphi$, $\operatorname{dom}\gamma$ and
$\operatorname{ran}\varphi$ are order-isomorphic to $\mathbb{Z}$, we
conclude that there exist bijective monotone maps
$\varphi_{\operatorname{dom}}\colon\operatorname{dom}\varphi
\rightarrow\mathbb{Z}$, $\gamma_{\operatorname{dom}}\colon
\operatorname{dom}\gamma \rightarrow\mathbb{Z}$ and
$\varphi_{\operatorname{ran}}\colon\operatorname{ran}\varphi
\rightarrow\mathbb{Z}$. We put
$\operatorname{dom}\kappa=\operatorname{dom}\gamma$,
$\operatorname{ran}\kappa=\operatorname{dom}\varphi$ and
$\kappa=\gamma_{\operatorname{dom}}
\cdot(\varphi_{\operatorname{dom}})^{-1}$. Then
$\kappa\colon\mathbb{Z}\rightharpoonup\mathbb{Z}$ is a monotone
partial map as a composition of monotone partial maps. We define an
injective partial map $\chi\colon\mathbb{Z}\rightharpoonup
\mathbb{Z}$ in the following way: $\operatorname{dom}\chi=
\mathbb{Z}$, $\operatorname{ran}\chi=\operatorname{ran}\gamma$ and
$(n)\chi=(n)\left((\varphi_{\operatorname{ran}})^{-1}\cdot
\varphi^{-1}\cdot\kappa\cdot\gamma\right)$ for $n\in\mathbb{Z}$.
Then $\chi\colon\mathbb{Z}\rightharpoonup\mathbb{Z}$ is a monotone
partial map, being a composition of monotone partial maps. We put
$\operatorname{dom}\xi=\operatorname{ran}\varphi$,
$\operatorname{ran}\xi=\operatorname{ran}\gamma$ and
$\xi=\varphi_{\operatorname{ran}}\cdot\chi$. Then
$\xi\colon\mathbb{Z}\rightharpoonup\mathbb{Z}$ is a monotone partial
map, being a composition of monotone partial maps. Hence
$\gamma=\kappa\cdot\varphi\cdot\xi$ and so
$\mathscr{I}^{\!\nearrow}_{\infty}(\mathbb{Z})$ is simple.

$(ix)$ Let $\varepsilon,\varphi\in
E\left(\mathscr{I}^{\!\nearrow}_{\infty}(\mathbb{Z})\right)$ be
arbitrary. Then by statement $(i)$ we have that
$\operatorname{dom}\varepsilon=\operatorname{ran}\varepsilon$ and
$\operatorname{dom}\varphi=\operatorname{ran}\varphi$. Since the
sets $\mathbb{Z}\setminus \operatorname{dom}\varepsilon$ and
$\mathbb{Z}\setminus \operatorname{dom}\varphi$ are finite and the
sets $\operatorname{dom}\varepsilon$ and $\operatorname{dom}\varphi$
are order-isomorphic to $\mathbb{Z}$ we conclude that there exist
bijective monotone maps
$\varepsilon_{\operatorname{dom}}\colon\operatorname{dom}\varepsilon
\rightarrow\mathbb{Z}$ and $\varphi_{\operatorname{dom}}\colon
\operatorname{dom}\varphi \rightarrow\mathbb{Z}$. Also, we note that
for every integer $k$ the translation $\sigma_k\colon\mathbb{Z}
\rightarrow\mathbb{Z}\colon n\mapsto n+k$ is a bijective monotone
map. Now we define for any integer $i$
\begin{equation*}
 \alpha_i=\varepsilon_{\operatorname{dom}}\cdot\sigma_i\cdot
 (\varphi_{\operatorname{dom}})^{-1}.
\end{equation*}
Then we have that
\begin{equation*}
\begin{split}
  \alpha_i\cdot\alpha_i^{-1}& =
    \varepsilon_{\operatorname{dom}}\cdot\sigma_i\cdot
    (\varphi_{\operatorname{dom}})^{-1}\cdot
    \varphi_{\operatorname{dom}}\cdot\sigma_i^{-1}\cdot
    (\varepsilon_{\operatorname{dom}})^{-1}=\\
  & =
    \varepsilon_{\operatorname{dom}}\cdot\sigma_i\cdot
    \mathbb{I}\cdot\sigma_i^{-1}\cdot
    (\varepsilon_{\operatorname{dom}})^{-1}=
    \varepsilon_{\operatorname{dom}}\cdot\sigma_i\cdot
    \sigma_i^{-1}\cdot
    (\varepsilon_{\operatorname{dom}})^{-1}=\\
  & =
    \varepsilon_{\operatorname{dom}}\cdot\mathbb{I}\cdot
    (\varepsilon_{\operatorname{dom}})^{-1}=
    \varepsilon_{\operatorname{dom}}\cdot
    (\varepsilon_{\operatorname{dom}})^{-1}=\\
  & =\varepsilon
\end{split}
\end{equation*}
and
\begin{equation*}
\begin{split}
  \alpha_i^{-1}\cdot\alpha_i& =
    \varphi_{\operatorname{dom}}\cdot\sigma_i^{-1}\cdot
    (\varepsilon_{\operatorname{dom}})^{-1}\cdot
    \varepsilon_{\operatorname{dom}}\cdot\sigma_i\cdot
    (\varphi_{\operatorname{dom}})^{-1}=\\
  & =
    \varphi_{\operatorname{dom}}\cdot\sigma_i^{-1}\cdot
    \mathbb{I}\cdot\sigma_i\cdot
    (\varphi_{\operatorname{dom}})^{-1}=
    \varphi_{\operatorname{dom}}\cdot\sigma_i^{-1}\cdot
    \sigma_i\cdot(\varphi_{\operatorname{dom}})^{-1}=\\
  & =
    \varphi_{\operatorname{dom}}\cdot\mathbb{I}\cdot
    (\varphi_{\operatorname{dom}})^{-1}=
    \varphi_{\operatorname{dom}}\cdot
    (\varphi_{\operatorname{dom}})^{-1}=\\
  & =
    \varphi
\end{split}
\end{equation*}
for every integer $i$. This completes the proof of the assertion.
\end{proof}

\begin{proposition}\label{proposition-2.2}
The group of units $H(\mathbb{I})$ of the semigroup
$\mathscr{I}^{\!\nearrow}_{\infty}(\mathbb{Z})$ is isomorphic to
$\mathbb{Z}(+)$.
\end{proposition}

\begin{proof}
Let $\alpha$ be an arbitrary element of $H(\mathbb{I})$. Then
$\alpha$ is a bijective monotone map from $\mathbb{Z}$ onto
$\mathbb{Z}$. We fix arbitrary $n\in\mathbb{Z}$. Then the
monotonicity of $\alpha$ implies that $(n)\alpha<(n+1)\alpha$. If
$(n)\alpha+1<(n+1)\alpha$ then there exists an integer $m$ such that
$(m)\alpha=(n)\alpha+1$. But if $m>n+1$ or $m<n$ this contradicts
the monotonicity of $\alpha$. Therefore we get that
$(n)\alpha+1=(n+1)\alpha$. Similarly we get that
$(n)\alpha-1=(n-1)\alpha$. Hence every $\alpha\in H(\mathbb{I})$ is
a shift of the set of integers. We define the map $h\colon
H(\mathbb{I})\rightarrow \mathbb{Z}(+)$ by the formula
$(\alpha)h=(n)\alpha-n$. Since $\alpha$ is a shift of the set of
integers we conclude that the definition of the map $h$ is correct.
Simple verifications show that $h\colon H(\mathbb{I})\rightarrow
\mathbb{Z}(+)$ is a group isomorphism.
\end{proof}

Since $\mathscr{I}^{\!\nearrow}_{\infty}(\mathbb{Z})$ is an inverse
semigroup, Proposition~\ref{proposition-2.1}~$(ix)$ and Lemma~1.1
from \cite{Munn1966} imply the following:

\begin{proposition}\label{proposition-2.3}
$\mathscr{I}^{\!\nearrow}_{\infty}(\mathbb{Z})$ is a bisimple
semigroup.
\end{proposition}

Theorem~2.20 of~\cite{CP}, and our
Propositions~\ref{proposition-2.2} and \ref{proposition-2.3} imply
the following corollary:

\begin{corollary}\label{corollary-2.4}
Every maximal subgroup of the semigroup
$\mathscr{I}^{\!\nearrow}_{\infty}(\mathbb{Z})$ is isomorphic to
$\mathbb{Z}(+)$.
\end{corollary}

\begin{proposition}\label{proposition-2.5}
For every
$\alpha,\beta\in\mathscr{I}^{\!\nearrow}_{\infty}(\mathbb{Z})$, both
sets
 $
\left\{\chi\in\mathscr{I}^{\!\nearrow}_{\infty}(\mathbb{Z})\mid
\alpha\cdot\chi=\beta\right\}
 $
 and
 $
\{\chi\in\mathscr{I}^{\!\nearrow}_{\infty}(\mathbb{Z})\mid
\chi\cdot\alpha=\beta\}
 $
are finite.
\end{proposition}

\begin{proof}
We denote
$A=\{\chi\in\mathscr{I}^{\!\nearrow}_{\infty}(\mathbb{Z})\mid
\alpha\cdot\chi=\beta\}$ and
$B=\{\chi\in\mathscr{I}^{\!\nearrow}_{\infty}(\mathbb{Z})\mid
\alpha^{-1}\cdot\alpha\cdot\chi=\alpha^{-1}\cdot\beta\}$. Then
$A\subseteq B$ and the restriction of any partial map $\chi\in B$ to
$\operatorname{dom}(\alpha^{-1}\cdot\alpha)$ coincides with the
partial map $\alpha^{-1}\cdot\beta$. Since every partial map from
$\mathscr{I}^{\!\nearrow}_{\infty}(\mathbb{Z})$ is monotone we
conclude that the set $B$ is finite and hence so is $A$.
\end{proof}

\begin{lemma}\label{lemma-2.6}
Let $S$ be an arbitrary semigroup and
$h\colon\mathscr{I}^{\!\nearrow}_{\infty}(\mathbb{Z}) \rightarrow S$
a semigroup homomorphism. If there exist distinct idempotents
$\varepsilon,
\varphi\in\mathscr{I}^{\!\nearrow}_{\infty}(\mathbb{Z})$ such that
$(\varepsilon)h=(\varphi)h$ then $(\psi)h=(\mathbb{I})h$ for all
$\psi\in E(\mathscr{I}^{\!\nearrow}_{\infty}(\mathbb{Z}))$.
\end{lemma}

\begin{proof}
Since $(\varepsilon)h=(\varphi)h=(\varphi\cdot\varphi)h=
(\varphi)h\cdot(\varphi)h=(\varphi)h\cdot(\varepsilon)h=
(\varphi\cdot\varepsilon)h$ we can assume without loss of generality
that $\varepsilon\leqslant\varphi$ in
$E(\mathscr{I}^{\!\nearrow}_{\infty}(\mathbb{Z}))$. Therefore, if
$\iota$ is an idempotent of the semigroup
$\mathscr{I}^{\!\nearrow}_{\infty}(\mathbb{Z})$ such that
$\varepsilon\leqslant\iota\leqslant\varphi$ then
$(\varepsilon)h=(\iota)h$. Hence
Proposition~\ref{proposition-2.1}~$(ii)$ implies that we can assume
without loss of generality that $|\operatorname{dom}\varphi\setminus
\operatorname{dom}\varepsilon|=1$.

Let $\psi$ let be an arbitrary idempotent of the semigroup
$\mathscr{I}^{\!\nearrow}_{\infty}(\mathbb{Z})$ and
$n_0=\min\{\mathbb{Z}\setminus\operatorname{dom}\psi\}-1$. Let
$\sigma\colon\mathbb{Z}\rightharpoonup\mathbb{Z}$ be a partial order
preserving injective map which maps $\operatorname{dom}\varphi$ onto
$\mathbb{Z}$ and
${n}=\mathbb{Z}\setminus(\operatorname{dom}\varepsilon)\sigma$.
Without loss of generality we can assume that $n=n_0$. Then
$\widetilde{\varphi}=\sigma^{-1}\circ\varphi\circ\sigma\colon\mathbb{Z}
\rightarrow\mathbb{Z}$ is an identity map and
$\widetilde{\varepsilon}=\sigma^{-1}\circ\varepsilon\circ\sigma$ is
an identity map from $\mathbb{Z}\setminus\{n_0\}$ onto
$\mathbb{Z}\setminus\{n_0\}$. Then $\widetilde{\varepsilon}$ is a
unit of the semigroup
$\mathscr{I}^{\!\nearrow}_{\infty}(\mathbb{Z})$. Since
$\sigma\in\mathscr{I}^{\!\nearrow}_{\infty}(\mathbb{Z})$ and
$\operatorname{dom}\varepsilon\varsubsetneqq\operatorname{dom}\varphi$
we have that
\begin{equation*}
(\widetilde{\varphi})h=(\sigma\cdot\varphi\cdot\sigma)h=
(\sigma)h\cdot(\varphi)h\cdot(\sigma)h=
(\sigma)h\cdot(\varepsilon)h\cdot(\sigma)h=
(\widetilde{\varepsilon})h,
\end{equation*}
\begin{equation*}
\begin{split}
\widetilde{\varepsilon}\cdot\widetilde{\varphi} &=
(\sigma^{-1}\cdot\varepsilon\cdot\sigma)\cdot
(\sigma^{-1}\cdot\varphi\cdot\sigma)=
\sigma^{-1}\cdot\varepsilon\cdot(\sigma\cdot
\sigma^{-1})\cdot\varphi\cdot\sigma=
\sigma^{-1}\cdot\varepsilon\cdot\varphi\cdot\varphi\cdot\sigma=\\
& =\sigma^{-1}\cdot\varepsilon\cdot\varphi\cdot\sigma=
\sigma^{-1}\cdot\varepsilon\cdot\sigma=\widetilde{\varepsilon},
\end{split}
\end{equation*}
\begin{equation*}
\begin{split}
  \widetilde{\varphi}\cdot\widetilde{\varepsilon} &
   = (\sigma^{-1}\cdot\varphi\cdot\sigma)\cdot
(\sigma^{-1}\cdot\varepsilon\cdot\sigma)=
\sigma^{-1}\cdot\varphi\cdot(\sigma\cdot
\sigma^{-1})\cdot\varepsilon\cdot\sigma=
\sigma^{-1}\cdot\varphi\cdot\varphi\cdot\varepsilon\cdot\sigma=\\
    & =
\sigma^{-1}\cdot\varphi\cdot\varepsilon\cdot\sigma=
\sigma^{-1}\cdot\varepsilon\cdot\sigma=\widetilde{\varepsilon},
\end{split}
\end{equation*}
and hence $\widetilde{\varepsilon}\leqslant\widetilde{\varphi}$.

We observe that $\widetilde{\varphi},\widetilde{\varepsilon}\in
\mathscr{C}(n_0,+)$. Since
$(\widetilde{\varphi})h=(\widetilde{\varepsilon})h$,
Corollary~1.32~\cite{CP} implies that
$(\widetilde{\varphi})h=(\chi)h$ for every idempotent
$\chi\in\mathscr{C}(n_0,+)$. Since $i>n_0$ for all
$i\in\mathbb{Z}\setminus\operatorname{dom}\psi$, the definition of
the semigroup $\mathscr{I}^{\!\nearrow}_{\infty}(\mathbb{Z})$
implies that there exists an idempotent $\varepsilon_0\in
\mathscr{C}(n_0,+)$ such that $\varepsilon_0\leqslant\psi\leqslant
\widetilde{\varphi}$. Therefore we have that
$(\psi)h=(\varepsilon_0)h=(\widetilde{\varphi})h$. This completes
the proof of the lemma.
\end{proof}

\begin{theorem}\label{theorem-2.7}
Let $S$ be a semigroup and
$h\colon\mathscr{I}_{\infty}^{\!\nearrow}(\mathbb{Z})\rightarrow S$
a non-annihilating homomorphism. Then either $h$ is a monomorphism
or $(\mathscr{I}_{\infty}^{\!\nearrow}(\mathbb{Z}))h$ is a subgroup
of $S$.
\end{theorem}

\begin{proof}
Suppose that
$h\colon\mathscr{I}_{\infty}^{\!\nearrow}(\mathbb{Z})\rightarrow S$
is not a monomorphism. Then $(\alpha)h=(\beta)h$, for some distinct
$\alpha,\beta\in\mathscr{I}_{\infty}^{\!\nearrow}(\mathbb{Z})$. We
consider two cases:
\begin{itemize}
  \item[$(i)$] $\alpha$ and $\beta$ are not
   $\mathscr{H}$-equivalent;
  \item[$(ii)$] $\alpha$ and $\beta$ are $\mathscr{H}$-equivalent.
\end{itemize}

Suppose that case $(i)$ holds. Since
$\mathscr{I}_{\infty}^{\!\nearrow}(\mathbb{Z})$ is an inverse
semigroup we have that either
$\alpha\cdot\alpha^{-1}\neq\beta\cdot\beta^{-1}$ or
$\alpha^{-1}\cdot\alpha\neq\beta^{-1}\cdot\beta$. Suppose that
$\alpha\cdot\alpha^{-1}\neq\beta\cdot\beta^{-1}$. In the other case
the proof is similar. Since
$\mathscr{I}_{\infty}^{\!\nearrow}(\mathbb{Z})$ is an inverse
semigroup we conclude that
\begin{equation*}
    (\alpha^{-1})h=\big((\alpha)h\big)^{-1}=
    \big((\beta)h\big)^{-1}=(\beta^{-1})h
\end{equation*}
and hence $(\alpha\cdot\alpha^{-1})h=(\alpha)h\cdot(\alpha^{-1})h=
(\beta)h\cdot(\beta^{-1})h=(\beta\cdot\beta^{-1})h$. Therefore the
assertion of Lemma~\ref{lemma-2.6} holds. Since every homomorphic
image of an inverse semigroup is an inverse semigroup we conclude
that $(\mathscr{I}_{\infty}^{\!\nearrow}(\mathbb{Z}))h$ is a
subgroup of $S$.

Suppose that $\alpha\mathscr{H}\beta$. Then by Theorem~2.20
of~\cite{CP} there exist distinct $\alpha_0,\beta_0\in
H(\mathbb{I})$ such that $(\alpha_0)h=(\beta_0)h$. Therefore we have
that $(\mathbb{I})h=(\gamma)h$ for
$\gamma=\alpha_0^{-1}\cdot\beta_0\in H(\mathbb{I})$ and
$\gamma\neq\mathbb{I}$. We fix an arbitrary integer $i$. Let
$\iota\colon\mathbb{Z}\setminus\{i\}\rightarrow
\mathbb{Z}\setminus\{i\}$ be an identity map. Then
$(\iota)h=(\iota\cdot\mathbb{I})h=(\iota\cdot\gamma)h$. Hence
$\iota$ is an idempotent of the semigroup
$\mathscr{I}_{\infty}^{\!\nearrow}(\mathbb{Z})$ and
$\operatorname{ran}\iota\neq\operatorname{ran}(\iota\cdot\gamma)$.
Therefore by Proposition~\ref{proposition-2.1}~$(vii)$ the elements
$\iota$ and $\iota\cdot\gamma$ are not $\mathscr{H}$-equivalent in
the semigroup $\mathscr{I}_{\infty}^{\!\nearrow}(\mathbb{Z})$. This
implies that there exist distinct non-$\mathscr{H}$-equivalent
elements  $\alpha,\beta$ in
$\mathscr{I}_{\infty}^{\looparrowright}(\mathbb{Z})$ such that
$(\alpha)h=(\beta)h$ and hence case $(i)$ holds. Therefore we get that
$(\mathscr{I}_{\infty}^{\looparrowright}(\mathbb{Z}))h$ is a
subgroup of $S$.
\end{proof}

\begin{proposition}\label{proposition-2.8}
Let $\mathfrak{C}_{mg}$ be a least group congruence on the semigroup
$\mathscr{I}^{\!\nearrow}_{\infty}(\mathbb{Z})$. Then the quotient
semigroup $\mathscr{I}^{\!\nearrow}_{\infty}(\mathbb{Z})
/\mathfrak{C}_{mg}$ is isomorphic to the direct product
$\mathbb{Z}(+)\times\mathbb{Z}(+)$.
\end{proposition}

\begin{proof}
Let $\alpha$ and $\beta$ be $\mathfrak{C}_{mg}$-equivalent elements
of the semigroup $\mathscr{I}^{\!\nearrow}_{\infty}(\mathbb{Z})$.
Then by Lemma~III.5.2 from \cite{Petrich1984} there exists an
idempotent $\varepsilon_0$ in
$\mathscr{I}^{\!\nearrow}_{\infty}(\mathbb{Z})$ such that
$\alpha\cdot\varepsilon_0=\beta\cdot\varepsilon_0$. Since
$\mathscr{I}^{\!\nearrow}_{\infty}(\mathbb{Z})$ is an inverse
semigroup we conclude that
$\alpha\cdot\varepsilon=\beta\cdot\varepsilon$ for all
$\varepsilon\in E(\mathscr{I}^{\!\nearrow}_{\infty}(\mathbb{Z}))$
such that $\varepsilon\leqslant\varepsilon_0$. Then
Lemma~\ref{lemma-1.1} implies that there exist integers $d$ and $u$
such that
\begin{align*}
    (m-1)\alpha&=(m)\alpha-1, &(n+1)\alpha&=(n)\alpha+1,\\
    (m-1)\beta&=(m)\beta-1,  &(n+1)\beta&=(n)\beta+1,
\end{align*}
for all integers $m\leqslant d$ and $n\geqslant u$. We put
$D=\min\{(d)\alpha,(d)\beta\}$ and $U=\max\{(u)\alpha,(u)\beta\}$.
Let $\varepsilon_1$ be an identity map from
$\mathbb{Z}\setminus\{D,D+1,\ldots,U\}$ onto itself. Then
$\varepsilon^0=\varepsilon_1\circ\varepsilon_0\leqslant
\varepsilon_0$ and hence we have that
$\alpha\cdot\varepsilon^0=\beta\cdot\varepsilon^0$. Therefore we
have showed that if the elements $\alpha$ and $\beta$ of the
semigroup $\mathscr{I}^{\!\nearrow}_{\infty}(\mathbb{Z})$ are
$\mathfrak{C}_{mg}$-equivalent, then there exist integers $d$ and
$u$ such that
\begin{equation*}
(m)\alpha=(m)\beta  \qquad \mbox{ and } \qquad (n)\alpha=(n)\beta,
\end{equation*}
for all integers $m\leqslant d$ and $n\geqslant u$.

Conversely, suppose that exist integers $d$ and $u$ such that
\begin{equation*}
(m)\alpha=(m)\beta  \qquad \mbox{ and } \qquad (n)\alpha=(n)\beta,
\end{equation*}
for all integers $m\leqslant d$ and $n\geqslant u$. Then we have
that $d\leqslant u$. If $d=u$ or $d=u-1$ then $\alpha=\beta$ in
$\mathscr{I}^{\!\nearrow}_{\infty}(\mathbb{Z})$ and hence $\alpha$
and $\beta$ are $\mathfrak{C}_{mg}$-equivalent. If $d<u-1$ then we
put $\varepsilon_0$ to be the identity map of the set
$\mathbb{Z}\setminus\{(d+1)\alpha,\ldots,(u-1)\alpha\}$. Then we get
that $(n)(\alpha\circ\varepsilon_0)=(n)(\beta\circ\varepsilon_0)$
for any $n\in\mathbb{Z}\setminus\{d+1,\ldots,u-1\}$ and therefore
$\alpha\cdot\varepsilon_0=\beta\cdot\varepsilon_0$. Hence
Lemma~III.5.2 from \cite{Petrich1984} implies that $\alpha$ and
$\beta$ are $\mathfrak{C}_{mg}$-equivalent elements of the semigroup
$\mathscr{I}^{\!\nearrow}_{\infty}(\mathbb{Z})$.

Now we define the map
$h\colon\mathscr{I}^{\!\nearrow}_{\infty}(\mathbb{Z})\rightarrow
\mathbb{Z}(+)\times\mathbb{Z}(+)$ by the formula
\begin{equation*}
    (\alpha)h=((d_\alpha)\alpha-d_\alpha,
    (u_\alpha)\alpha-u_\alpha),
\end{equation*}
where the integers $d_\alpha$ and $u_\alpha$ are defined in
Lemma~\ref{lemma-1.1}.

We observe that
\begin{equation*}
    (d_\alpha-n)\alpha=(d_\alpha)\alpha-n \qquad \hbox{ and } \qquad
    (u_\alpha+n)\alpha=(u_\alpha)\alpha+n,
\end{equation*}
for any positive integer $n$. Hence we have that
\begin{equation*}
    (m)\alpha-m=(d_\alpha)\alpha-d_\alpha \qquad \hbox{ and } \qquad
    (n)\alpha-n=(u_\alpha)\alpha-u_\alpha,
\end{equation*}
for all integers $m\leqslant d_\alpha$ and $n\geqslant u_\alpha$.

Lemma~\ref{lemma-1.1} implies that there exist integers $d^0$ and
$u^0$ such that
\begin{align*}
    (m-1)\alpha&=(m)\alpha-1, &(n+1)\alpha&=(n)\alpha+1,\\
    (m-1)\beta&=(m)\beta-1,  &(n+1)\beta&=(n)\beta+1,\\
    (m-1)(\alpha\cdot\beta)&=(m)(\alpha\cdot\beta)-1,
         &(n+1)(\alpha\cdot\beta)&=(n)(\alpha\cdot\beta)+1,
\end{align*}
for all integers $m\leqslant d^0$ and $n\geqslant u^0$. Hence for
all integers $m\leqslant d^0$ and $n\geqslant u^0$ we have that
\begin{equation*}
    (m)(\alpha\cdot\beta)-m=
    (m)(\alpha\cdot\beta)-(m)\alpha+(m)\alpha-m=
    ((d_\beta)\beta-d_\beta)+((d_\alpha)\alpha-d_\alpha),
\end{equation*}
\begin{equation*}
    (n)(\alpha\cdot\beta)-n=
    (n)(\alpha\cdot\beta)-(n)\alpha+(n)\alpha-n=
    ((u_\beta)\beta-u_\beta)+((u_\alpha)\alpha-u_\alpha).
\end{equation*}
This implies that the map
$h\colon\mathscr{I}^{\!\nearrow}_{\infty}(\mathbb{Z})\rightarrow
\mathbb{Z}(+)\times\mathbb{Z}(+)$ is a homomorphism.
\end{proof}

Theorem~\ref{theorem-2.7} and Proposition~\ref{proposition-2.8}
imply the following:

\begin{theorem}\label{theorem-2.9}
Let $S$ be a semigroup and
$h\colon\mathscr{I}^{\!\nearrow}_{\infty}(\mathbb{Z})\rightarrow S$
a non-annihilating homomorphism. Then either $h$ is a monomorphism
or $(\mathscr{I}^{\!\nearrow}_{\infty}(\mathbb{Z}))h$ is a
homomorphic image of the group $\mathbb{Z}(+)\times\mathbb{Z}(+)$.
\end{theorem}

\section{Some remarks on the semigroup
$\mathscr{I}^{\looparrowright}_{\infty}(\mathbb{Z})$}

In this section we shall denote the identity of the semigroup
$\mathscr{I}^{\looparrowright}_{\infty}(\mathbb{Z})$ by $\mathbb{I}$
and the group of units of
$\mathscr{I}^{\looparrowright}_{\infty}(\mathbb{Z})$ by
$H(\mathbb{I})$. The proof of the following proposition is similar
to corresponding propositions in Section~\ref{section-2}.

\begin{proposition}\label{proposition-2.1a}
\begin{itemize}
    \item[$(i)$] $E(\mathscr{I}^{\looparrowright}_{\infty}(\mathbb{Z}))=
     E(\mathscr{I}^{\!\nearrow}_{\infty}(\mathbb{Z}))$.

    \item[$(ii)$] $\alpha\mathscr{R}\beta$ in
         $\mathscr{I}^{\looparrowright}_{\infty}(\mathbb{Z})$ if
         and only if
         $\operatorname{dom}\alpha=\operatorname{dom}\beta$.

    \item[$(iii)$] $\alpha\mathscr{L}\beta$ in
         $\mathscr{I}^{\looparrowright}_{\infty}(\mathbb{Z})$ if
         and only if
         $\operatorname{ran}\alpha=\operatorname{ran}\beta$.

    \item[$(iv)$] $\alpha\mathscr{H}\beta$ in
         $\mathscr{I}^{\looparrowright}_{\infty}(\mathbb{Z})$ if
         and only if
         $\operatorname{dom}\alpha=\operatorname{dom}\beta$ and
         $\operatorname{ran}\alpha=\operatorname{ran}\beta$.

    \item[$(v)$] $\mathscr{I}^{\looparrowright}_{\infty}(\mathbb{Z})$
         is a simple semigroup and hence
         $\mathscr{J}=\mathscr{I}^{\looparrowright}_{\infty}(\mathbb{Z})
         \times\mathscr{I}^{\looparrowright}_{\infty}(\mathbb{Z})$.

     \item[$(vi)$] For all idempotents $\varepsilon,\varphi\in
         \mathscr{I}^{\looparrowright}_{\infty}(\mathbb{Z})$ there exist
         infinitely many elements $\alpha,\beta\in
         \mathscr{I}^{\looparrowright}_{\infty}(\mathbb{Z})$ such that
         $\alpha\cdot\beta=\varepsilon$ and
         $\beta\cdot\alpha=\varphi$.

    \item[$(vii)$] $\mathscr{I}^{\looparrowright}_{\infty}(\mathbb{Z})$
     is a bisimple semigroup.

    \item[$(viii)$] For all
     $\alpha,\beta\in\mathscr{I}^{\looparrowright}_{\infty}(\mathbb{Z})$,
     both sets
     $\left\{\chi\in\mathscr{I}^{\looparrowright}_{\infty}(\mathbb{Z})\mid
     \alpha\cdot\chi=\beta\right\}$ and
     $\left\{\chi\in\mathscr{I}^{\looparrowright}_{\infty}(\mathbb{Z})\mid
     \chi\cdot\alpha=\beta\right\}$ are finite.
\end{itemize}
\end{proposition}

\begin{proposition}\label{proposition-2.2a}
For every
$\alpha\in\mathscr{I}^{\looparrowright}_{\infty}(\mathbb{Z})$ there
exist idempotents $\varepsilon_l,\varepsilon_r,\varepsilon$ in
$\mathscr{I}^{\looparrowright}_{\infty}(\mathbb{Z})$ such that
$\varepsilon_l\cdot\alpha, \alpha\cdot\varepsilon_r,
\varepsilon\cdot\alpha\cdot\varepsilon\in
\mathscr{I}^{\!\nearrow}_{\infty}(\mathbb{Z})$.
\end{proposition}

\begin{proof}
The definition of the semigroup
$\mathscr{I}^{\looparrowright}_{\infty}(\mathbb{Z})$ implies that
for every element $\alpha$ of
$\mathscr{I}^{\looparrowright}_{\infty}(\mathbb{Z})$ there exists a
smallest finite (or empty) subset $F_\alpha$ such that the
restriction $\alpha|_{\operatorname{dom}\alpha\setminus
F_\alpha}\colon\mathbb{Z} \rightharpoonup\mathbb{Z}$ is a monotone
partial map. We put
$\varepsilon_l=\operatorname{id}_{\operatorname{dom}\alpha\setminus
F_\alpha}$ to be the identity map from
$\operatorname{dom}\alpha\setminus F_\alpha$ onto
$\operatorname{dom}\alpha\setminus F_\alpha$. Also we set
$\varepsilon_r=\operatorname{id}_{(\operatorname{dom}\alpha\setminus
F_\alpha)\alpha}$ and $\varepsilon=\varepsilon_l\cdot\varepsilon_r$.
Then we have that $\varepsilon_l\cdot\alpha,
\alpha\cdot\varepsilon_r, \varepsilon\cdot\alpha\cdot\varepsilon\in
\mathscr{I}^{\!\nearrow}_{\infty}(\mathbb{Z})$.
\end{proof}

We denote by $\emph{\textsf{S}}_\infty(\mathbb{Z})$ the group of all
bijective transformations of $\mathbb{Z}$ with finite supports
(i.e., $\alpha\in\emph{\textsf{S}}_\infty(\mathbb{Z})$ if and only
if the set $\{x\in\mathbb{Z}\mid(x)\alpha\neq x\}$ is finite). We
observe that $\emph{\textsf{S}}_\infty(\mathbb{Z})$ is a subgroup of
the group of units $H(\mathbb{I})$ of the semigroup
$\mathscr{I}^{\looparrowright}_{\infty}(\mathbb{Z})$ and since every
element $\alpha$ in $H(\mathbb{I})$ is an almost monotone bijective
selfmap of the set of integers we get that for every $\alpha\in
H(\mathbb{I})$ there exists an integer $n_\alpha$ such that the set
$\{i\in\mathbb{Z}\mid(i)\alpha+n_\alpha\neq i\}$ is finite. This
observation implies that $\emph{\textsf{S}}_\infty(\mathbb{Z})$ is a
normal subgroup of $H(\mathbb{I})$. Moreover, we have that every
element of the group of units $H(\mathbb{I})$ has a unique
representation $\alpha=\sigma\cdot\beta$ by the formula
$(n)\alpha=(n)\sigma+\beta$, $n\in\mathbb{Z}$, where
$\sigma\in\alpha\in\emph{\textsf{S}}_\infty(\mathbb{Z})$ and
$\beta\in\mathbb{Z}(+)$. Hence we have that $H(\mathbb{I})=
\emph{\textsf{S}}_\infty(\mathbb{Z})\cdot \mathbb{Z}(+)$ and it is
obvious that $\emph{\textsf{S}}_\infty(\mathbb{Z})\cap
\mathbb{Z}(+)=\{\mathbb{I}\}$. Thus the group $\mathbb{Z}(+)$ acts
on $\emph{\textsf{S}}_\infty(\mathbb{Z})$ by the conjugation action
in $H(\mathbb{I})$ and hence it follows by Exercise~2.5.3 from
\cite{Dixon-Mortimer1996} that the group $H(\mathbb{I})$ is
isomorphic to the semidirect product
$\emph{\textsf{S}}_\infty(\mathbb{Z})\rtimes\mathbb{Z}(+)$ (or split
extension of $\emph{\textsf{S}}_\infty(\mathbb{Z})$ by
$\mathbb{Z}(+)$). Also, we observe that since the action of the
group $\mathbb{Z}(+)$ on $\emph{\textsf{S}}_\infty(\mathbb{Z})$ is
not the identity map we conclude that the group $H(\mathbb{I})$ is
not isomorphic to the direct product
$\emph{\textsf{S}}_\infty(\mathbb{Z})\times\mathbb{Z}(+)$. We put
$\widetilde{\mathbb{Z}}(+)=\emph{\textsf{S}}_\infty(\mathbb{Z})
\rtimes\mathbb{Z}(+)$. Therefore we have proved the following:

\begin{proposition}\label{proposition-2.3a}
The group of units $H(\mathbb{I})$ of the semigroup
$\mathscr{I}^{\looparrowright}_{\infty}(\mathbb{Z})$ is isomorphic
to $\widetilde{\mathbb{Z}}(+)$.
\end{proposition}

Proposition~\ref{proposition-2.1a}~$(vii)$ and Theorem~2.20
of~\cite{CP} imply the following corollary:

\begin{corollary}\label{corollary-2.4a}
Every maximal subgroup of the semigroup
$\mathscr{I}^{\looparrowright}_{\infty}(\mathbb{Z})$ is isomorphic
to $\widetilde{\mathbb{Z}}(+)$.
\end{corollary}

\begin{theorem}\label{theorem-2.5a}
Let $S$ be a semigroup and
$h\colon\mathscr{I}_{\infty}^{\looparrowright}(\mathbb{Z})\rightarrow
S$ a non-annihilating homomorphism. Then either $h$ is a
monomorphism or
$(\mathscr{I}_{\infty}^{\looparrowright}(\mathbb{Z}))h$ is a
subgroup of $S$.
\end{theorem}

\begin{proof}
Suppose that
$h\colon\mathscr{I}_{\infty}^{\looparrowright}(\mathbb{Z})\rightarrow
S$ is not an monomorphism. Then $(\alpha)h=(\beta)h$, for some
distinct $\alpha,\beta\in
\mathscr{I}_{\infty}^{\looparrowright}(\mathbb{Z})$.

Suppose that $\alpha$ and $\beta$ are not $\mathscr{H}$-equivalent.
Since $\mathscr{I}_{\infty}^{\looparrowright}(\mathbb{Z})$ is an
inverse semigroup, we have that either
$\alpha\cdot\alpha^{-1}\neq\beta\cdot\beta^{-1}$ or
$\alpha^{-1}\cdot\alpha\neq\beta^{-1}\cdot\beta$. Suppose that
$\alpha\cdot\alpha^{-1}\neq\beta\cdot\beta^{-1}$. Since
$\mathscr{I}_{\infty}^{\looparrowright}(\mathbb{Z})$ is an inverse
semigroup, we conclude that $ (\alpha^{-1})h=(\beta^{-1})h$ and
hence $(\alpha\cdot\alpha^{-1})h=(\beta\cdot\beta^{-1})h$. Therefore
the assertion of Lemma~\ref{lemma-2.6} holds for the subsemigroup
$\mathscr{I}_{\infty}^{\!\nearrow}(\mathbb{Z})$ of the semigroup
$\mathscr{I}_{\infty}^{\looparrowright}(\mathbb{Z})$. Now by
Proposition~\ref{proposition-2.1a}~$(i)$ and since every homomorphic
image of an inverse semigroup is an inverse semigroup it follows
that $(\mathscr{I}_{\infty}^{\looparrowright}(\mathbb{Z}))h$ is a
subgroup of $S$.

Suppose that $\alpha\mathscr{H}\beta$. Then by Theorem~2.20
of~\cite{CP} there exist distinct $\alpha_0,\beta_0\in
H(\mathbb{I})$ such that $(\alpha_0)h=(\beta_0)h$. Then we have that
$(\mathbb{I})h=(\gamma)h$ for $\gamma=\alpha_0^{-1}\cdot\beta_0\in
H(\mathbb{I})$ and $\gamma\neq\mathbb{I}$. We fix an arbitrary
integer $i$. Let $\iota\colon\mathbb{Z}\setminus\{i\}\rightarrow
\mathbb{Z}\setminus\{i\}$ be the identity map. Hence
$(\iota)h=(\iota\cdot\mathbb{I})h=(\iota\cdot\gamma)h$. Then $\iota$
is an idempotent of the semigroup
$\mathscr{I}_{\infty}^{\looparrowright}(\mathbb{Z})$ and
$\operatorname{ran}\iota\neq\operatorname{ran}(\iota\cdot\gamma)$.
Therefore by Proposition~\ref{proposition-2.1a}~$(iv)$ the elements
$\iota$ and $\iota\cdot\gamma$ are not $\mathscr{H}$-equivalent in
the semigroup $\mathscr{I}_{\infty}^{\looparrowright}(\mathbb{Z})$.
This implies that there exist distinct non-$\mathscr{H}$-equivalent
elements  $\alpha,\beta$ in
$\mathscr{I}_{\infty}^{\looparrowright}(\mathbb{Z})$ such that
$(\alpha)h=(\beta)h$ and hence
$(\mathscr{I}_{\infty}^{\looparrowright}(\mathbb{Z}))h$ is a
subgroup of $S$.
\end{proof}

\begin{proposition}\label{proposition-2.6a}
Let $\mathfrak{C}_{mg}$ be a least group congruence on the semigroup
$\mathscr{I}^{\looparrowright}_{\infty}(\mathbb{Z})$. Then the
quotient semigroup
$\mathscr{I}^{\looparrowright}_{\infty}(\mathbb{Z})
/\mathfrak{C}_{mg}$ is isomorphic to the direct product
$\mathbb{Z}(+)\times\mathbb{Z}(+)$.
\end{proposition}

\begin{proof}
Let $\alpha$ and $\beta$ be $\mathfrak{C}_{mg}$-equivalent elements
of the semigroup
$\mathscr{I}^{\looparrowright}_{\infty}(\mathbb{Z})$. Then by
Lemma~III.5.2 from \cite{Petrich1984} there exists an idempotent
$\varepsilon_0$ in
$\mathscr{I}^{\looparrowright}_{\infty}(\mathbb{Z})$ such that
$\alpha\cdot\varepsilon_0=\beta\cdot\varepsilon_0$. By
Proposition~\ref{proposition-2.2a} we can assume without loss of
generality that $\alpha\cdot\varepsilon_0,\beta\cdot\varepsilon_0\in
\mathscr{I}_{\infty}^{\!\nearrow}(\mathbb{Z})$. Then similarly as in
the proof of Proposition~\ref{proposition-2.8} we can show that
elements $\alpha$ and $\beta$ of the semigroup
$\mathscr{I}^{\looparrowright}_{\infty}(\mathbb{Z})$ are
$\mathfrak{C}_{mg}$-equivalent if and only if there exist integers
$d$ and $u$ such that
\begin{equation*}
(m)\alpha=(m)\beta  \qquad \mbox{ and } \qquad (n)\alpha=(n)\beta,
\end{equation*}
for all integers $m\leqslant d$ and $n\geqslant u$.

Let $\alpha_0=\alpha\cdot\varepsilon_0$. Then the map
$h\colon\mathscr{I}^{\looparrowright}_{\infty}(\mathbb{Z})\rightarrow
\mathbb{Z}(+)\times\mathbb{Z}(+)$ defined by the formula
\begin{equation*}
    (\alpha)h=((d_{\alpha_0})\alpha-d_{\alpha_0},
    (u_{\alpha_0})\alpha-u_{\alpha_0}),
\end{equation*}
where the integers $d_{\alpha_0}$ and $u_{\alpha_0}$ are defined for
element $\alpha_0$ of the semigroup
$\mathscr{I}^{\!\nearrow}_{\infty}(\mathbb{Z})$ in
Lemma~\ref{lemma-1.1}, is a natural homomorphism which is generated
by the least group congruence $\mathfrak{C}_{mg}$ on the semigroup
$\mathscr{I}^{\looparrowright}_{\infty}(\mathbb{Z})$.
\end{proof}

Theorem~\ref{theorem-2.5a} and Proposition~\ref{proposition-2.6a}
imply the following:

\begin{theorem}\label{theorem-2.7a}
Let $S$ be a semigroup and
$h\colon\mathscr{I}^{\looparrowright}_{\infty}(\mathbb{Z})\rightarrow
S$ a non-annihilating homomorphism. Then either $h$ is a
monomorphism or
$(\mathscr{I}^{\looparrowright}_{\infty}(\mathbb{Z}))h$ is a
homomorphic image of the group $\mathbb{Z}(+)\times\mathbb{Z}(+)$.
\end{theorem}

\section{On topologizations of the semigroup
$\mathscr{I}^{\!\nearrow}_{\infty}(\mathbb{Z})$}

\begin{theorem}\label{theorem-3.1}
Every Baire topology $\tau$ on
$\mathscr{I}^{\!\nearrow}_{\infty}(\mathbb{Z})$ such that
$(\mathscr{I}^{\!\nearrow}_{\infty}(\mathbb{Z}),\tau)$ is a
Hausdorff semitopological semigroup is discrete.
\end{theorem}

\begin{proof}
If no point in $\mathscr{I}^{\!\nearrow}_{\infty}(\mathbb{Z})$ is
isolated,  then since
$(\mathscr{I}^{\!\nearrow}_{\infty}(\mathbb{Z}),\tau)$ is Hausdorff,
it follows that $\{\alpha\}$ is nowhere dense for all
$\alpha\in\mathscr{I}^{\!\nearrow}_{\infty}(\mathbb{Z})$. But, if
this is the case, then since
$\mathscr{I}^{\!\nearrow}_{\infty}(\mathbb{Z})$ is countable it
cannot be a Baire space. Hence
$\mathscr{I}^{\!\nearrow}_{\infty}(\mathbb{Z})$ contains an isolated
point $\mu$. If
$\gamma\in\mathscr{I}^{\!\nearrow}_{\infty}(\mathbb{Z})$ is
arbitrary, then by Proposition~\ref{proposition-2.1}~$(viii)$, there
exist $\alpha,\beta\in
\mathscr{I}^{\!\nearrow}_{\infty}(\mathbb{Z})$ such that
$\alpha\cdot\gamma\cdot\beta=\mu$. The map
$f\colon\chi\mapsto\alpha\cdot\chi\cdot\beta$ is continuous and so
$(\{\mu\})f^{-1}$ is open. By Proposition~\ref{proposition-2.5},
$(\{\mu\})f^{-1}$ is finite and since
$(\mathscr{I}^{\!\nearrow}_{\infty}(\mathbb{Z}),\tau)$ is Hausdorff,
$\{\gamma\}$ is open, and hence isolated.
\end{proof}

Since every \v{C}ech complete space (and hence every locally compact
space) is Baire, Theorem~\ref{theorem-3.1} implies
Corollaries~\ref{corollary-3.2} and \ref{corollary-3.3}.

\begin{corollary}\label{corollary-3.2}
Every Hausdorff \v{C}ech complete (locally compact) topology $\tau$
on $\mathscr{I}^{\!\nearrow}_{\infty}(\mathbb{Z})$ such that
$(\mathscr{I}^{\!\nearrow}_{\infty}(\mathbb{Z}),\tau)$ is a
Hausdorff semitopological semigroup is discrete.
\end{corollary}

\begin{corollary}\label{corollary-3.3}
Every Hausdorff Baire topology (and hence \v{C}ech complete or
locally compact topology) $\tau$ on
$\mathscr{I}^{\!\nearrow}_{\infty}(\mathbb{Z})$ such that
$(\mathscr{I}^{\!\nearrow}_{\infty}(\mathbb{Z}),\tau)$ is a
Hausdorff topological semigroup is discrete.
\end{corollary}

The following example shows that there exists a non-discrete
Tychonoff topology $\tau_W$ on the semigroup
$\mathscr{I}^{\!\nearrow}_{\infty}(\mathbb{Z})$ such that
$(\mathscr{I}^{\!\nearrow}_{\infty}(\mathbb{Z}),\tau_W)$ is a
topological inverse semigroup.

\begin{example}\label{example-3.4}
We define a topology $\tau_{W}$ on the semigroup
$\mathscr{I}^{\!\nearrow}_{\infty}(\mathbb{Z})$ as follows. For
every $\alpha\in\mathscr{I}^{\!\nearrow}_{\infty}(\mathbb{Z})$ we
define a family
\begin{equation*}
    \mathscr{B}_{W}(\alpha)=\left\{U_\alpha(F)\mid F \mbox{ is a finite
    subset of } \operatorname{dom}\alpha\right\},
\end{equation*}
where
\begin{equation*}
    U_\alpha(F)=
    \left\{\beta\in\mathscr{I}^{\!\nearrow}_{\infty}(\mathbb{Z})
    \mid \operatorname{dom}\beta\subseteq\operatorname{dom}\alpha
    \mbox{ and }
    (x)\beta=(x)\alpha \mbox{ for all } x\in F\right\}.
\end{equation*}
It is straightforward to verify that
$\{\mathscr{B}_{W}(\alpha)\}_{\alpha\in
\mathscr{I}^{\!\nearrow}_{\infty}(\mathbb{Z})}$ forms a basis for a
topology $\tau_{W}$ on the semigroup
$\mathscr{I}^{\!\nearrow}_{\infty}(\mathbb{Z})$.
\end{example}

\begin{proposition}\label{proposition-3.5}
$(\mathscr{I}^{\!\nearrow}_{\infty}(\mathbb{Z}),\tau_{W})$ is a
Tychonoff topological inverse semigroup.
\end{proposition}

\begin{proof}
Let $\alpha$ and $\beta$ be arbitrary elements of the semigroup
$\mathscr{I}^{\!\nearrow}_{\infty}(\mathbb{Z})$. We put
$\gamma=\alpha\cdot\beta$ and let $F=\{n_1,\ldots,n_i\}$ be a finite
subset of $\operatorname{dom}\gamma$. We denote
$m_1=(n_1)\alpha,\ldots,m_i=(n_i)\alpha$ and
$k_1=(n_1)\gamma,\ldots,k_i=(n_i)\gamma$. Then we get that
$(m_1)\beta=k_1,\ldots,(m_i)\beta=k_i$. Hence we have that
\begin{equation*}
    U_\alpha(\{n_1,\ldots,n_i\})\cdot
    U_\beta(\{m_1,\ldots,m_i\})\subseteq
    U_\gamma(\{n_1,\ldots,n_i\})
\end{equation*}
and
\begin{equation*}
    \big(U_\gamma(\{n_1,\ldots,n_i\})\big)^{-1}\subseteq
    U_{\gamma^{-1}}(\{k_1,\ldots,k_i\}).
\end{equation*}
Therefore the semigroup operation and the inversion are continuous
in $(\mathscr{I}^{\!\nearrow}_{\infty}(\mathbb{Z}),\tau_{W})$.

Let $Z=\mathbb{Z}\cup\{a\}$ for some $a\notin\mathbb{Z}$. Then $Z^Z$
with the operation composition is a semigroup and the map
$\Psi\colon\mathscr{I}^{\!\nearrow}_{\infty}(\mathbb{Z})\rightarrow
Z^Z$ defined by the formula
\begin{equation*}
    (x)(\alpha)\Psi=
\left\{
  \begin{array}{ll}
    (x)\alpha, & \hbox{if } x\in\operatorname{dom}\alpha; \\
    a,         & \hbox{if } x\notin\operatorname{dom}\alpha
  \end{array}
\right.
\end{equation*}
is a monomorphism. Hence $Z^Z$ is a topological semigroup with the
product topology if $Z$ has the discrete topology. Obviously, this
topology generates topology $\tau_W$ on
$\mathscr{I}^{\!\nearrow}_{\infty}(\mathbb{Z})$. Therefore by
Theorem~2.3.11 from \cite{Engelking1989} topological space $Z^Z$ is
Tychonoff and hence by Theorem~2.1.6 from \cite{Engelking1989} so is
$(\mathscr{I}^{\!\nearrow}_{\infty}(\mathbb{Z}),\tau_{W})$. This
completes the proof of the proposition.
\end{proof}

\begin{theorem}\label{theorem-3.6}
Let $S$ be a topological semigroup which contains
$\mathscr{I}^{\!\nearrow}_{\infty}(\mathbb{Z})$ as a dense discrete
subsemigroup. If
$I=S\setminus\mathscr{I}^{\!\nearrow}_{\infty}(\mathbb{Z})
\neq\varnothing$ then $I$ is an ideal of $S$.
\end{theorem}

\begin{proof}
Suppose that $I$ is not an ideal of $S$. Then at least one of the
following conditions holds:
\begin{equation*}
    1)~I\cdot\mathscr{I}^{\!\nearrow}_{\infty}(\mathbb{Z})\nsubseteq I,
    \qquad 2)~\mathscr{I}^{\!\nearrow}_{\infty}(\mathbb{Z})\cdot
    I\nsubseteq I,
    \qquad \mbox{or}
    \qquad 3)~I\cdot I\nsubseteq I.
\end{equation*}
Since $\mathscr{I}^{\!\nearrow}_{\infty}(\mathbb{Z})$ is a dense
discrete subspace of $S$, Theorem~3.5.8 from~\cite{Engelking1989}
implies that $\mathscr{I}^{\!\nearrow}_{\infty}(\mathbb{Z})$ is an
open subspace of $S$. Suppose there exist
$\alpha\in\mathscr{I}^{\!\nearrow}_{\infty}(\mathbb{Z})$ and
$\beta\in I$ such that $\beta\cdot\alpha=\gamma\notin I$. Since
$\mathscr{I}^{\!\nearrow}_{\infty}(\mathbb{Z})$ is a dense open
discrete subspace of $S$, the continuity of the semigroup operation
in $S$ implies that there exists an open neighbourhood $U(\beta)$ of
$\beta$ in $S$ such that $U(\beta)\cdot \{\alpha\}=\{\gamma\}$.
Hence we have that
$\big(U(\beta)\cap\mathscr{I}^{\!\nearrow}_{\infty}(\mathbb{Z})\big)
\cdot \{\alpha\}=\{\gamma\}$ and the set
$U(\beta)\cap\mathscr{I}^{\!\nearrow}_{\infty}(\mathbb{Z})$ is
infinite. But by Proposition~\ref{proposition-2.5}, the equation
$\chi\cdot\alpha=\gamma$ has finitely many solutions in
$\mathscr{I}^{\!\nearrow}_{\infty}(\mathbb{Z})$. This contradicts
the assumption that $\beta\in
S\setminus\mathscr{I}^{\!\nearrow}_{\infty}(\mathbb{Z})$. Therefore
$\beta\cdot\alpha=\gamma\in I$ and hence
$I\cdot\mathscr{I}^{\!\nearrow}_{\infty}(\mathbb{Z})\subseteq I$.
The proof of the inclusion
$\mathscr{I}^{\!\nearrow}_{\infty}(\mathbb{Z})\cdot I\subseteq I$ is
similar.

Suppose there exist $\alpha,\beta\in I$ such that $\alpha\cdot
\beta=\gamma\notin I$. Since
$\mathscr{I}^{\!\nearrow}_{\infty}(\mathbb{Z})$ is a dense open
discrete subspace of $S$, the continuity of the semigroup operation
in $S$ implies that there exist open neighbourhoods $U(\alpha)$ and
$U(\beta)$ of $\alpha$ and $\beta$ in $S$, respectively, such that
$U(\alpha)\cdot U(\beta)=\{\gamma\}$. Hence we have that
$\big(U(\beta)\cap\mathscr{I}^{\!\nearrow}_{\infty}(\mathbb{Z})\big)
\cdot
\big(U(\alpha)\cap\mathscr{I}^{\!\nearrow}_{\infty}(\mathbb{Z})\big)
=\{\gamma\}$ and the sets
$U(\beta)\cap\mathscr{I}^{\!\nearrow}_{\infty}(\mathbb{Z})$ and
$U(\alpha)\cap\mathscr{I}^{\!\nearrow}_{\infty}(\mathbb{Z})$ are
infinite. But by Proposition~\ref{proposition-2.5}, the equations
$\chi\cdot\beta=\gamma$ and $\alpha\cdot\kappa=\gamma$ have finitely
many solutions in $\mathscr{I}^{\!\nearrow}_{\infty}(\mathbb{Z})$.
This contradicts the assumption that $\alpha,\beta\in
S\setminus\mathscr{I}^{\!\nearrow}_{\infty}(\mathbb{Z})$. Therefore
$\alpha\cdot\beta=\gamma\in I$ and hence $I\cdot I\subseteq I$.
\end{proof}

\begin{proposition}\label{proposition-3.7}
Let $S$ be a Hausdorff topological semigroup which contains
$\mathscr{I}^{\!\nearrow}_{\infty}(\mathbb{Z})$ as a dense discrete
subsemigroup. Then for every
$\gamma\in\mathscr{I}^{\!\nearrow}_{\infty}(\mathbb{Z})$ the set
\begin{equation*}
    D_\gamma=\left\{(\chi,\varsigma)\in
    \mathscr{I}^{\!\nearrow}_{\infty}(\mathbb{Z})
    \times\mathscr{I}^{\!\nearrow}_{\infty}(\mathbb{Z})\mid
    \chi\cdot \varsigma=\gamma\right\}
\end{equation*}
is a closed-and-open subset of $S\times S$.
\end{proposition}

\begin{proof}
Since $\mathscr{I}^{\!\nearrow}_{\infty}(\mathbb{Z})$ is a discrete
subspace of $S$ we have that $D_\gamma$ is an open subset of
$S\times S$.

Suppose that there exists
$\gamma\in\mathscr{I}^{\!\nearrow}_{\infty}(\mathbb{Z})$ such that
$D_\gamma$ is a non-closed subset of $S\times S$. Then there exists
an accumulation point $(\alpha,\beta)\in S\times S$ of the set
$D_\gamma$. The continuity of the semigroup operation in $S$ implies
that $\alpha\cdot\beta=\gamma$. But
$\mathscr{I}^{\!\nearrow}_{\infty}(\mathbb{Z})\times
\mathscr{I}^{\!\nearrow}_{\infty}(\mathbb{Z})$ is a discrete
subspace of $S\times S$ and hence by Theorem~\ref{theorem-3.6}, the
points $\alpha$ and $\beta$ belong to the ideal $I=S\setminus
\mathscr{I}^{\!\nearrow}_{\infty}(\mathbb{Z})$ and hence
$\alpha\cdot \beta\in
S\setminus\mathscr{I}^{\!\nearrow}_{\infty}(\mathbb{Z})$ cannot be
equal to $\gamma$.
\end{proof}

\begin{theorem}\label{theorem-3.8}
If a Hausdorff topological semigroup $S$ contains
$\mathscr{I}^{\!\nearrow}_{\infty}(\mathbb{Z})$ as a dense discrete
subsemigroup then the square $S\times S$ cannot be pseudocompact.
\end{theorem}

The proof of Theorem~\ref{theorem-3.8} is similar to that of
Theorem~5.1$(3)$ of \cite{{BanakhDimitrovaGutik2010}}.

Recall that, a topological semigroup $S$ is called $\Gamma$-compact
if for every $x\in S$ the closure of the set $\{x,x^2,x^3,\ldots\}$
is a compactum in $S$ (see \cite{HildebrantKoch1986}). We recall
that the Stone-\v{C}ech compactification of a Tychonoff space $X$ is
a compact Hausdorff space $\beta X$ containing $X$ as a dense
subspace so that each continuous map $f\colon X\rightarrow Y$ to a
compact Hausdorff space $Y$ extends to a continuous map
$\overline{f}\colon \beta X\rightarrow Y$ \cite{Engelking1989}.

\begin{corollary}\label{corollary-3.9}
If a topological semigroup $S$ satisfies one of the following
conditions: $(i)$~$S$ is compact; $(ii)$~$S$ is $\Gamma$-compact;
$(iii)$~the square $S\times S$ is countably compact; $(iv)$~$S$ is a
countably compact topological inverse semigroup; or $(v)$~the square
$S\times S$ is a Tychonoff pseudocompact space, then $S$ does not
contain the semigroup
$\mathscr{I}^{\!\nearrow}_{\infty}(\mathbb{Z})$ (and hence the
semigroup $\mathscr{I}^{\looparrowright}_{\infty}(\mathbb{Z})$).
\end{corollary}

\begin{proof}
By Theorem~2 from \cite{Koch-Wallace1957} every compact
topological semigroup is stable. But by Corollary~3.1 of \cite{AHK}
a stable semigroup cannot contain the bicyclic semigroup. Since by
Remark~\ref{remark-1.2}, for any positive integer $n$ the semigroup
$\mathscr{I}^{\!\nearrow}_{\infty}(\mathbb{Z})$ contains the
semigroup $\mathscr{C}(n,+)$ which is isomorphic to the bicyclic
semigroup we conclude that any compact topological semigroup cannot
contain the semigroup
$\mathscr{I}^{\!\nearrow}_{\infty}(\mathbb{Z})$.

Similarly by Proposition~5.3 of \cite{HildebrantKoch1986} no
$\Gamma$-compact topological semigroup can contain the bicyclic
semigroup. Also the proof of Theorem~10 from
\cite{BanakhDimitrovaGutik2009} implies that every topological
semigroup $S$ with countably compact square $S\times S$ cannot
contain the bicyclic semigroup and by Theorem~1 from
\cite{GutikRepovs2007} any countably compact topological inverse
semigroup cannot contain the bicyclic semigroup, either. Next we
apply Remark~\ref{remark-1.2}.

By Theorem~1.3 from \cite{BanakhDimitrova2010} for any topological
semigroup $S$ with the pseudocompact square $S\times S$ the
semigroup operation $\mu\colon S\times S\rightarrow S$ extends to a
continuous semigroup operation $\beta\mu\colon \beta S\times\beta
S\rightarrow\beta S$, so S is a subsemigroup of the compact
topological semigroup $\beta S$. Therefore if $S$ contains the
bicyclic semigroup then $\beta S$ also contains the bicyclic
semigroup which is a contradiction.
\end{proof}

The proofs of the following three theorems are similar to
Theorems~\ref{theorem-3.1}, \ref{theorem-3.6} and \ref{theorem-3.8},
respectively.

\begin{theorem}\label{theorem-3.10}
Every Baire topology $\tau$ on
$\mathscr{I}^{\looparrowright}_{\infty}(\mathbb{Z})$ such that
$(\mathscr{I}^{\looparrowright}_{\infty}(\mathbb{Z}),\tau)$ is a
Hausdorff semitopological semigroup is discrete.
\end{theorem}

\begin{theorem}\label{theorem-3.11}
Let $S$ be a topological semigroup which contains
$\mathscr{I}^{\looparrowright}_{\infty}(\mathbb{Z})$ as a dense
discrete subsemigroup. If
$I=S\setminus\mathscr{I}^{\looparrowright}_{\infty}(\mathbb{Z})
\neq\varnothing$ then $I$ is an ideal of $S$.
\end{theorem}

\begin{theorem}\label{theorem-3.12}
If a Hausdorff topological semigroup $S$ contains
$\mathscr{I}^{\looparrowright}_{\infty}(\mathbb{Z})$ as a dense
discrete subsemigroup then the square $S\times S$ cannot be
pseudocompact.
\end{theorem}

\begin{remark}\label{remark-3.13}
We observe that the topology $\tau^\looparrowright_W$ on the
semigroup $\mathscr{I}^{\looparrowright}_{\infty}(\mathbb{Z})$ which
is generated by the family
\begin{equation*}
    \mathscr{B}^\looparrowright_{W}(\alpha)=\left\{U_\alpha(F)\mid F
    \mbox{ is a finite
    subset of } \operatorname{dom}\alpha\right\},
\end{equation*}
where
\begin{equation*}
    U_\alpha(F)=\left\{\beta\in
    \mathscr{I}^{\looparrowright}_{\infty}(\mathbb{Z})
    \mid \operatorname{dom}\beta\subseteq\operatorname{dom}\alpha
    \mbox{ and }
    (x)\beta=(x)\alpha \mbox{ for all } x\in F\right\},
\end{equation*}
is a non-discrete inverse semigroup topology. The proof of
continuity of the semigroup operation and inversion in
$(\mathscr{I}^{\looparrowright}_{\infty}(\mathbb{Z}),
\tau^\looparrowright_W)$ is similar to the proof of
Proposition~\ref{proposition-3.5} and obviously the topology
$\tau^\looparrowright_W$ induces the topology $\tau_W$ on the
subsemigroup $\mathscr{I}^{\!\nearrow}_{\infty}(\mathbb{Z})$.
\end{remark}

The following example shows that there exists a non-discrete
Tychonoff topology $\tau^\looparrowright_H$ on the semigroup
$\mathscr{I}^{\looparrowright}_{\infty}(\mathbb{Z})$ such that
$(\mathscr{I}^{\looparrowright}_{\infty}(\mathbb{Z}),
\tau^\looparrowright_H)$ is a topological inverse semigroup, every
$H$-class in $\mathscr{I}^{\looparrowright}_{\infty}(\mathbb{Z})$ is
an open subset in
$(\mathscr{I}^{\looparrowright}_{\infty}(\mathbb{Z}),
\tau^\looparrowright_H)$ and the topology $\tau^\looparrowright_H$
is finer than the topology $\tau^\looparrowright_W$.

\begin{example}\label{example-3.14}
We define a topology $\tau^\looparrowright_{H}$ on the semigroup
$\mathscr{I}^{\looparrowright}_{\infty}(\mathbb{Z})$ as follows. For
every $\alpha\in\mathscr{I}^{\looparrowright}_{\infty}(\mathbb{Z})$
we define a family
\begin{equation*}
    \mathscr{B}^\looparrowright_{H}(\alpha)=\left\{W_\alpha(F)\mid F
    \mbox{ is a finite subset of } \operatorname{dom}\alpha\right\},
\end{equation*}
where
\begin{equation*}
    W_\alpha(F)=\left\{\beta\in
    \mathscr{I}^{\looparrowright}_{\infty}(\mathbb{Z})
    \mid \beta\mathscr{H}\alpha  \mbox{ and }
    (x)\beta=(x)\alpha \mbox{ for all } x\in F\right\}.
\end{equation*}
It is straightforward to verify that
$\{\mathscr{B}^\looparrowright_{H}(\alpha)\}_{\alpha\in
\mathscr{I}^{\looparrowright}_{\infty}(\mathbb{Z})}$ forms a basis
for a topology $\tau^\looparrowright_{H}$ on the semigroup
$\mathscr{I}^{\looparrowright}_{\infty}(\mathbb{Z})$.
\end{example}

\begin{proposition}\label{proposition-3.15}
$(\mathscr{I}^{\looparrowright}_{\infty}(\mathbb{Z}),
\tau^\looparrowright_{H})$ is a Tychonoff topological inverse
semigroup.
\end{proposition}

\begin{proof}
The proof of continuity of the semigroup operation and inversion in
$(\mathscr{I}^{\looparrowright}_{\infty}(\mathbb{Z}),
\tau^\looparrowright_H)$ is similar to the proof of
Proposition~\ref{proposition-3.5}. Also the definition of the
topology $\tau^\looparrowright_{H}$ implies that all $H$-classes are
open subsets in
$(\mathscr{I}^{\looparrowright}_{\infty}(\mathbb{Z}),
\tau^\looparrowright_H)$ and the group of units $H(\mathbb{I})$ of
the semigroup $\mathscr{I}^{\looparrowright}_{\infty}(\mathbb{Z})$
with the induced topology from
$(\mathscr{I}^{\looparrowright}_{\infty}(\mathbb{Z}),
\tau^\looparrowright_H)$ is a non-discrete topological group, and
hence by Theorem~II.8.4 from \cite{HewittRoos1963} the topological
subspace $H(\mathbb{I})$ is Tychonoff. Hence since every $H$-class
in topological inverse semigroup $S$ is a closed subset in $S$ (see
\cite{EberhartSelden1969}), Proposition~\ref{proposition-2.1a},
Corollary~\ref{corollary-2.4a} and Theorem~2.20 of \cite{CP} imply
that the topological space
$(\mathscr{I}^{\looparrowright}_{\infty}(\mathbb{Z}),
\tau^\looparrowright_H)$ is homeomorphic to topological sum of a
countable many of topological spaces $H(\mathbb{I})$, and hence
$(\mathscr{I}^{\looparrowright}_{\infty}(\mathbb{Z}),
\tau^\looparrowright_H)$ is a Tychonoff space.
\end{proof}

We observe that the topology $\tau^\looparrowright_{H}$ on the
semigroup $\mathscr{I}^{\looparrowright}_{\infty}(\mathbb{Z})$
induces the discrete topology on its subsemigroup
$\mathscr{I}^{\!\nearrow}_{\infty}(\mathbb{Z})$.

Recall~\cite{DeLeeuwGlicksberg1961} that a \emph{Bohr
compactification} of a topological semigroup $S$ is a~pair $(\beta,
B(S))$ such that $B(S)$ is a compact topological semigroup,
$\beta\colon S\to B(S)$ is a continuous homomorphism, and if
$g\colon S\to T$ is a continuous homomorphism of $S$ into a compact
semigroup $T$, then there exists a unique continuous homomorphism
$f\colon B(S)\to T$ such that the diagram
\begin{equation*}
\xymatrix{ S\ar[rr]^\beta\ar[dr]_g && B(S)\ar[ld]^f\\
& T &}
\end{equation*}
commutes. Then Theorems~\ref{theorem-2.9} and \ref{theorem-2.7a},
and Proposition~2 from \cite{AndersonHunter1969a} imply the
following:

\begin{corollary}\label{corollary-3.16}
The Bohr compactification of the discrete semigroup
$\mathscr{I}^{\!\nearrow}_{\infty}(\mathbb{Z})$
$\left(\mathscr{I}^{\looparrowright}_{\infty}(\mathbb{Z})\right)$ is
topologically isomorphic to the Bohr compactification of discrete
group $\mathbb{Z}(+)\times\mathbb{Z}(+)$.
\end{corollary}

\section*{Acknowledgements}

This research was supported by the Slovenian Research Agency grants
P1-0292-0101,  J1-4144-0101 and BI-UA/11-12-001. The authors are
grateful to the referee for useful comments and suggestions.


\end{document}